\documentclass{amsart}
\usepackage[utf8]{inputenc}
\usepackage{amsmath,amssymb,amscd,amsthm,amsxtra,mathtools,hyperref,color}

\title{The Density of Elliptic Dedekind Sums}
\author[N. Berkopec]{Nicolas  Berkopec}
\address{University of New Mexico, Albuquerque NM, USA}
\email{nberkopec@unm.edu}
\author[J. Branch]{Jacob Branch}
\address{Fairmont State University, Fairmont WV, USA}
\email{jbranch@students.fairmontstate.edu}
\author[R. Heikkinen]{Rachel Heikkinen}
\address{Augustana College, Rock Island Il, USA}
\email{rachelheikkinen19@augustana.edu}
\author[C. Nunn]{Caroline Nunn}
\address{University of Wisconsin-Madison, Madison WI, USA}
\email{cnunn@wisc.edu}
\author[T.A. Wong]{Tian An Wong}
\address{University of Michigan, Dearborn MI, USA}
\email{tiananw@umich.edu}


\newtheorem{thrm}{Theorem}
\newtheorem{lem}{Lemma}
\theoremstyle{remark}
\newtheorem{rem}{Remark}

\newcommand{\legendre}[2]{\genfrac(){.1pt}{0}{#1}{#2}}
\newcommand{\pmat}[4]{\begin{pmatrix}#1 & #2\\#3 & #4\end{pmatrix}}
\DeclareMathOperator{\sign}{sign}

\begin{document}

\begin{abstract}
    Elliptic Dedekind sums were introduced by R. Sczech as generalizations  of  classical Dedekind sums to complex lattices. We show that for any lattice with real $j$-invariant, the values of suitably normalized elliptic Dedekind sums are dense in the real numbers. This extends an earlier result of Ito for Euclidean imaginary quadratic rings. Our proof is an adaptation of the recent work of Kohnen, which gives a new proof of the density of values of classical Dedekind sums.
\end{abstract}

\maketitle

\section{Introduction}
\subsection{The density of classical Dedekind sums} 
The classical Dedekind sum $s(m,n)$ is defined for $m,n\in\mathbb{Z}$, $(m,n)=1$, $n\neq 0$, by
\[
s(m,n):=\frac{1}{4n}\sum_{k=1}^{n-1}\cot\left(\pi\frac{mk}{n}\right)\cot\left(\pi\frac{k}{n}\right).
\]
These sums were first studied by Dedekind due to their appearance in the transformation law for the logarithm of the Dedekind $\eta$ function.  Specifically, for integers $a$, $b$, $c$, and $d$ with $c\neq0$ and $ad-bc=1$,
$$\log\eta\Big(\frac{a\tau+b}{c\tau+d}\Big)=\log\eta(\tau)+\frac{1}{2}\log\Big(\frac{c\tau+d}{i\ \sign (c)}\Big)+\frac{\pi i}{12}\phi\begin{pmatrix}a & b\\ c&d \end{pmatrix},$$
where $\phi$ is a map from $SL_2(\mathbb{Z})$ to $\mathbb{Z}$ given by
$$\phi\pmat{a}{b}{c}{d}=\begin{cases}{\frac{b}{d}} &  \text{if }c=0 \\  
\frac{a+d}{c}-12(\sign( c))s(d,|c|))&   \text{if }c\not=0 \end{cases},
 $$
 and $\tau$ is any point in the upper half-plane. The logarithm is taken with respect to the principle branch.  The $\phi$ map is a quasi-homomorphism satisfying
 \[\phi(A_1)=\phi(A_2)+\phi(A_3)-3\sign(c_1c_2c_3),\]
 where $A_1=A_2A_3$ and $c_i$ denotes the bottom left entry of $A_i$.  This near homomorphism law gives rise to many interesting results in the study of Dedekind sums and makes them an important tool in number theory.
 %A result central to the study of Dedekind sums is Rademacher's\cite{radegross} reciprocity law $s(p, q) + s(q, p) = \frac{1}{12}(\frac{p}{q} + \frac{1}{pq} + \frac{q}{p}) - \frac{1}{4}.$
%Using the $\phi$ homomorphism and assuming the conditions $(n,b)=(c,d)=1$, Girstmair\cite{girstmair} generalized the two term relation to a three term relation,
%$$12s(n,b) =12s(c,d) +\sign(q)12s(r,q)+\frac{b^2+d^2+q^2}{bdq}-3\sign(q).$$
%where $q:=nd-bc$, $j$ and $k$ are solutions to $dk-cj=1$ and $r=nj-bk$. 

It was conjectured by Grosswald and Rademacher that the values of $s(m,n)$ are dense in $\mathbb R$, and moreover the graph $(m/n, s(m,n))$ is dense in $\mathbb R^2$ \cite{radegross}. The latter statement (which implies the former) was proved by Hickerson \cite{hick}. Recently, Kohnen gave a new proof of the former \cite{kohnen}, using a 
result of Girstmair  \cite{girstmair} on the three-term relation of Dedekind sums and Dirichlet's theorem on arithmetic progressions of primes, providing a direct approximation of any rational number.

\subsection{The density of elliptic Dedekind sums} 

In this paper, we use Kohnen's  method to provide an equivalent result for normalized elliptic Dedekind sums for any imaginary quadratic field. %Consider an imaginary quadratic field $K$ with discriminant $D$ and ring of integers $\mathcal{O}_K$.  Pick a lattice $L$  such that $\mathcal{O}_K=\{m \in \mathbb{C} :mL\subset L\}$. For any positive integer $k$, we define the Eisenstein-Kronecker series
Let $L$ be a non-degenerate lattice in $\mathbb C$.  We define
$$
E_k(z)=\sum_{\substack{l\in L,\\ l+z\not= 0}} (l+z)^{-k}|l+z|^{-s}\Big|_{s=0},
$$
where the value of the sum at $s=0$ is evaluated by means of analytic continuation. We remark on the similarity between these elliptic functions and the cotangent function used to define classical Dedekind sums. 

Define $\mathcal O_L:=\{m \in \mathbb{C} :mL\subset L\}$.  This set is known as the ring of multipliers for $L$ and is either equal to the integers or to an order in an imaginary quadratic field (see exercise 12.3 of Neukirch \cite{neukirch} for reference).  Then, following Sczech \cite{sczech}, the elliptic Dedekind sums for $L$ are defined as %Elliptic Dedekind sums, for a lattice $L$, were defined by Sczech \cite{sczech} for $h$, $k$ in $\mathcal O_K$,
$$
D_L(h,k)=\frac{1}{k}\sum_{\mu \in L/kL} E_1\Big(\frac{h\mu}k\Big)E_1\Big(\frac{\mu}k\Big)
$$
for $h,k\in\mathcal O_L$, $k \neq 0$. It is important to note that $D_L(h,k)$ is dependent on the choice of the lattice $L$. %For the sake of clarity, we will suppress the choice of lattice in our notation and implicitly assume that our choice of lattice is fixed.

As with the classical case,  we define a map $\Phi$ from $SL_2(\mathcal{O}_L)$ to $\mathbb{C}$ given by 
$$\Phi\begin{pmatrix} a&b\\ c&d \end{pmatrix} := 
\begin{cases} 
      E_2(0)I\left(\frac{a+d}{c}\right)-D_L(a,c) & c\not = 0 \\
      E_2(0)I\left(\frac{b}{d}\right) & c=0
   \end{cases},
$$
where $I(z) = z - \overline{z}$.  It was shown by Sczech \cite{sczech} that $\Phi$ is a homomorphism in the additive group of complex numbers and that it is trivial for $\mathcal{O}_L=\mathbb{Z},\mathbb{Z}[i],$ and $\mathbb{Z}[\rho]$, where $\rho = (-1+\sqrt{-3})/2$. 

For our purposes it is advantageous to follow Ito \cite{ito} and define normalized elliptic Dedekind sums as
$$
\tilde D_L(a,c)=(i\sqrt{|d_L|}E_2(0))^{-1}D_L(a,c).
$$
Our departure from Ito is in our choice of $d_L$. In particular, any ring of multipliers $\mathcal{O}_L$ is an order of some imaginary quadratic field $K=\mathbb{Q}(\sqrt{-n})$ with a ring of integers having discriminant $d_K$. Recall that the ring of integers is the maximal order in $K$. Hence, we may write $\mathcal{O}_L=\mathbb{Z}[f\frac{d_K+\sqrt{d_K}}{2}]$ (for this result see Neukirch \cite{neukirch}). We therefore take $d_L= f^2d_K$ with a positive integer $f$, called the conductor of $\mathcal O_L$, instead of only using the discriminant $d_K$ of $K$.

Ito \cite{ito} proved that the values of the normalized elliptic Dedekind sums are dense in $\mathbb R$ when $\mathcal{O}_L$ is the ring of integers of $\mathbb{Q}(\sqrt{-2})$, $\mathbb{Q}(\sqrt{-5})$, or $\mathbb{Q}(\sqrt{-7})$.  In fact, for these cases, Ito showed that the set $(a/c, \tilde D_L(a,c))$ is dense in $\mathbb C\times \mathbb R$, using the continued fraction method of Hickerson \cite{hick}. As continued fraction expansions are less well-developed for non-Euclidean rings, especially for $\mathcal{O}_L$ with class number greater than $1$, Ito's method does not immediately generalize. In any case, Ito's result implies that $\tilde{D}(a,c)$ is dense in $\mathbb R$ for the cases listed above.

Using the method introduced by Kohnen \cite{kohnen} for the case of classical Dedekind sums, we provide an alternate proof that the normalized elliptic Dedekind sums are dense in the real line, and we show that the density result holds more generally, for every lattice $L$ with a real j-invariant $j(L)$.

\begin{thrm}
\label{main}
Let $L$ be a lattice with $j$-invariant $j(L)\in\mathbb R$ and multiplier ring $\mathcal O_L=\mathbb{Z}[f\frac{d_K+\sqrt{d_K}}{2}]$ as above, and $\mathcal O_L$ different from $\mathbb{Z},\mathbb{Z}[i]$, and $\mathbb{Z}[\frac{1-\sqrt{-3}}{2}]$. Under these conditions, the values of the normalized elliptic Dedekind sum $\tilde D_L(a,c)$ with respect to $L$ are dense in $\mathbb R$.
\end{thrm}

%\noindent 
%Again recall that because any ring of multipliers $\mathcal{O}_L$ is an order of some imaginary quadratic field $K=\mathbb{Q}(\sqrt{-n})$ with discriminant $d_K$. We may write it as $\mathcal{O}_L=\mathbb{Z}[f\frac{d_K+\sqrt{d_K}}{2}]$ for $f\geq 0$.

\noindent Note that if the lattice $L$ is similar to $\mathcal O_L$, then $j(L)$ is real. In particular, this includes all imaginary quadratic ring of integers $\mathcal O_K$. 

We shall first prove in Section \ref{3term} a lemma that provides us with an elliptic Dedekind sum identity analogous to the three term relation of Girstmair.  After proving the lemma, we will be prepared to show that $\tilde D_L(a,c)$ is dense in the real numbers in Section \ref{mainsection}. We will use a subset of the rationals that is dense in the reals with the idea of approximating values in this subset. The theorem then follows as an application of Dirichlet's theorem and quadratic reciprocity.

\section{An application of the three-term relation}
%We begin the proof of the density result with a Lemma that uses the $\Phi$ function. Sczech[5] says that for three matrices $A_i \in SL_2(\mathcal{O}_L)$, if $A_1A_2A_3 = 1$, then $\Phi(A_1) + \Phi(A_2) + \Phi(A_3) = 0$. We set $c\not= 0$ and then use the definition of $\Phi$ to create a three-term relation. Then, like Kohnen[3], based on the conditions of the matrices, two of the Dedekind sums cancel. Left over is one Dedekind sum in an equation which we simplify using properties of matrices in $SL_2(\mathcal{O}_L)$. This Lemma is necessary for the following Theorem to give a simple expression for the elliptic Dedekind sum that we can then substitute our fixed values.

\label{3term}
We first prove a lemma that establishes a simple formula for certain values of elliptic Dedekind sums.  This lemma makes use of the homomorphism property of $\Phi$, and follows the method of Girstmair \cite{girstmair} and Kohnen \cite{kohnen}.

\begin{lem}\label{onetermrel} Let
$A_i:= \begin{pmatrix}
a_i & b_i\\
c_i & d_i
\end{pmatrix} \in SL_2(\mathcal{O}_L)$ for $ i\in\{1,2,3\}$.  If $A_1=A_2A_3$ with the conditions $c_1=c_2=c\not=0$ and $a_1a_2\equiv 1 \ \pmod c$ satisfied, then
$$D_L(a_3,c_3)=E_2(0)I\left(\frac{2}{c_3}+\frac{c_3}{c^2}\right).$$
\end{lem}

\begin{proof} 
 Let $A_i$ be as above.  Since $A_i\in SL_2(\mathcal{O}_L)$, the following equations must hold
 \begin{align}
a_2d_2-b_2c=1, \label{eq1}\\
a_1d_1-b_1c=1. \label{eq2}
\end{align}
In Section 1.2 of \cite{sczech}, Sczech states that $\Phi$, as defined above, is a homomorphism from the multiplicative group $SL_2(\mathcal{O}_L)$ to the additive group of complex numbers. Therefore, 
$$
0=\Phi(\text{Id})=\Phi(A_1^{-1}A_2A_3)=-\Phi(A_1)+\Phi(A_2)+\Phi(A_3).
$$
After expanding and reordering terms as needed, we arrive at the three term relation
\[
D_L(a_1,c)=D_L(a_2,c)+D_L(a_3,c_3)+E_2(0)I\left(\frac{a_1+d_1}{c}-\frac{a_2+d_2}{c}-\frac{a_3+d_3}{c_3}\right).
\]
Fix $c_1=c_2$  and $a_1a_2\equiv 1 \ (c)$.  We claim that $D_L(a_1,c_1)=D_L(a_2,c_2)$. This follows from
\[D_L(a_1,c)=\frac{1}{c}\sum_{k\in L/cL}E_1\left(\frac{a_1k}{c}\right)E_1\left(\frac{k}{c}\right).\]
With the change of variables $k'=a_1k$, this becomes
\[
\frac{1}{c}\sum_{k'\in L/cL}E_1\left(\frac{k'}{c}\right)E_1\left(\frac{a_2k'}{c}\right)=D_L(a_2,c).
\]
Since $D_L(a_1,c)=D_L(a_2,c)$, we can conclude that
\[
D_L(a_3,c_3)=E_2(0)I\left(\frac{a_2+d_2}{c}+\frac{a_3+d_3}{c_3}-\frac{a_1+d_1}{c}\right).
\]
Now note that by assumption $A_3=A_2^{-1}A_1$, so we may write 
$$A_3=\begin{pmatrix}
a_1d_2-b_2c & d_2b_1-d_1b_2\\
a_2c-a_1c & d_1a_2-b_1c
\end{pmatrix}.
$$
Hence,
\begin{align*}
&\frac{a_2+d_2}{c}+\frac{a_3+d_3}{c_3}-\frac{a_1+d_1}{c}\\ &=\frac{(a_2-a_1)(a_2+d_2)}{(a_2-a_1)c}+\frac{a_1d_2-b_2c +d_1a_2-b_1c}{(a_2-a_1)c} -\frac{(a_2-a_1)(a_1+d_1)}{(a_2-a_1)c},
\end{align*}
which by \eqref{eq1} and \eqref{eq2} reduces to
$$\frac{2+a_2^2-2a_1a_2+a_1^2}{(a_2-a_1)c}=\frac{2}{c_3}+\frac{c_3}{c^2}.$$ 
Therefore, by substitution we arrive at
\[
D_L(a_3,c_3)=E_2(0)I\left(\frac{2}{c_3}+\frac{c_3}{c^2}\right) 
\]
as desired.
\end{proof}

\section{Proof of Theorem \ref{main}} 

\label{mainsection}

We first explain the choice of the normalization of $\tilde D_L(h,k)$ that we employ. It is worth noting that all values taken by normalized elliptic Dedekind sums over imaginary quadratic fields are real, provided $j(L) \in \mathbb{R}$, and in fact are real algebraic. Ito \cite{ito} states that normalized elliptic Dedekind sums over imaginary quadratic fields take values in $\mathbb{Q}(j)$, where $j$ is equal to the $j$-invariant of the lattice $L$. Since we assumed $j$ to be real in the statement of Theorem \ref{main}, $\mathbb Q(j)$ is a subset of the real numbers.  In the important case where $L$ is similar to its own multiplier ring $O_L$ in an imaginary quadratic field, we can easily see that $j(L)$ is real. Recall that $j$ is defined by
\[j(\tau)=1728\frac{g_2(\tau)^3}{\Delta(\tau)},\]
where 
\[
\Delta(\tau)=g_2(\tau)^3-27g_3(\tau)^2
\]
and $g_2,g_3$ are Eisenstein series. The $q=e^{2\pi i \tau}$ expansions of these series have purely real coefficients. Since \[\tau=f\frac{d_K+\sqrt{d_K}}{2},\] we see that $q=e^{2\pi i \tau}$ is real if and only if $fd_K$ is an integer. The last condition is true by assumptions on $f$ and $d_K$. Therefore, $j(\tau)$ and the normalized elliptic Dedekind sums are also real. 

One can show more generally that $j(L)$ is real if and only if the similarity class of $L$ is invariant under complex conjugation. Since complex conjugation acts on ideal classes by taking an ideal class into its inverse, the condition $j(L) \in \mathbb{R}$ means that the order of the ideal class represented by the lattice $L$ is at most 2 in the class group of $O_L$.

With this, we are prepared to prove density. In the following proof we recall the convention $d:=d_L= f^2d_K$ introduced above.

\begin{proof}[Proof of Theorem \ref{main}]
Since the rationals are dense in the reals, it suffices to show that range of $\tilde D_L$ is dense in $\{\frac{a}{b}\in\mathbb{Q}:(b,2d)=1\}$, which is a dense subset of the rationals.  Let $x=\frac{a}{b}$ be a rational number written in lowest terms with $(b,2d)=1$. Let $\bar a$ be the inverse of $a\pmod{b}$.  
%To follow Kohnen's proof, we want to set $a_1 := k \sqrt{d}$ and $a_2 := (k + e) \sqrt{d}$. However, in order to apply the Lemma, we must first show $a_1a_2\equiv 1\pmod{p}$.
%\[k(k+e)d\equiv 1\pmod{p}\]
%\[4dk^2+4dek\equiv 4\pmod{p}\]
%\[4dk^2+4dek+de^2\equiv 4+de^2\pmod{p}\]
%\[d^2(2k+e)^2=4d+d^2e^2\pmod{p}\]
%We need $\legendre{4d + d^2e^2}{p} = 1$ by quadratic reciprocity, so first substitute $e:=\frac{ap-1}{b}$
%\[d^2e^2+4d=d^2(\frac{ap-1}{b})^2+4d\pmod{p}\]
%this simplifies to
%\[\frac{d^2+4b^2d}{b^2}\pmod{p}\]
%The denominator $b^2$ will always be a square mod $p$, so we only have to show that the numerator $d^2 +4b^2d$ is a square mod $p$.
%In order to apply Lemma 1, we will choose $c=p$, and have to choose $a_1$ and $a_2$ satisfying $a_1a_2 \equiv 1 \pmod{p}$. Before our proof of Theorem 1, we check that for some $k$, $k(k+e)d \equiv 1\pmod{p}$. This statement is equivalent to $d^2(2k+e)^2 \equiv 4d+de^2 \pmod{p}$. Now we want $4d+de^2 \equiv 1 \pmod {p}$ in order to get a square mod $p$. After substituting in $e$, we only have to show that the numerator $d^2+4db^2$ is a square mod $p$ using quadratic reciprocity.
The assumptions on $a$ and $b$ allow us to find an infinite sequences of primes $p$ satisfying the requirement that  $d^2e^2+4d$ be a square mod $p$, where $e:=(ap-1)/b.$

 Indeed, since $(b,2d)=1$, $b$ is odd and shares no common factors with $d$.  So $(b,4(4b^2d+d^2))=(b,4b^2d+d^2)=(b,d)=1$.  By Dirichlet's theorem on primes in arithmetic progressions and the Chinese Remainder Theorem, there exist infinitely many primes $p$ such that
\[
p\equiv1\pmod{4(4b^2d+d^2)},\qquad p\equiv\bar{a}\pmod{b},
\]
where $\bar a $ is the inverse of $a$ mod $b$. The first congruence condition guarantees that $p\equiv1$ (mod $4$) and $p\equiv1$ (mod $4b^2d+d^2$).  As above, define $e:=(ap-1)/b$. This is an integer by the second congruence condition.  Then 
\begin{align*}
\legendre{d^2e^2+4d}{p}&=\legendre{d^2+4b^2d}{p}\\
  &=\legendre{p}{d^2+4b^2d}\\
  &=\legendre{1}{d^2+4b^2d}\\
  &=1.
\end{align*}
Therefore, $d^2e^2+4d$ is a square (mod $p$).  So there exists an $\ell$ with $(\ell, p) = 1$ such that
\[
d^2e^2+4d\equiv(2\ell-de)^2\pmod{p},
\]
which simplifies to
\[
(1+\ell e)d\equiv\ell^2\pmod{p}.
\]
Let $k$ be the inverse of $\ell \pmod{p}$. Then $k(k+e)d\equiv 1\pmod{p}$.
We set 
\[
a_1:=k\sqrt{d}\quad \text{and}\quad a_2:=(k+e)\sqrt{d}.
\]
Then $(a_i,p)=1$, and 
\[
a_1a_2=k(k+e)d\equiv1 \pmod{p}.
\]
Since $(p,k)=(p,d)=1$, we have $(p,kd)=1$.  Therefore, there exist integers $x_1,y_1$ such that $px_1+kdy_1=1$, and thus 
\[
px_1+a_1(y_1\sqrt d)=1.
\]
Similarly, there exist integers $x_2,y_2$ such that $px_2+a_2y_2\sqrt{d}=1$.  Let 
\[A_1=\begin{pmatrix}
a_1 & -x_1\\
p & y_1\sqrt d
\end{pmatrix}, \qquad A_2=\begin{pmatrix}
a_2 & -x_2\\
p & y_2\sqrt d
\end{pmatrix},
\]
and
\[
A_3=\begin{pmatrix}a_3 & b_3\\
c_3 & d_3\end{pmatrix}=A_2^{-1}A_1.
\]
By construction, each of these matrices are in $SL_2(\mathcal O_L)$.  So we may apply Lemma \ref{onetermrel} to find
\begin{align*}
D_L(a_3,c_3)&=E_2(0)I\left(\frac{2}{pe\sqrt{d}}+\frac{e\sqrt d}{p}\right)\\
  &=E_2(0)2\frac{\sqrt d}{d}\left(\frac{2}{ep}+\frac{e}{p}\right).
\end{align*}
Thus,
\[
\tilde D_L({a_3,c_3})=\frac{2}{ep}+\frac{e}{p},
\]
and by substituting $e=(ap-1)/{b}$, we have
\[
\frac{2b}{p(ap-1)}+\frac{ap-1}{bp}.
\]
Taking the limit as $p\rightarrow\infty$ and using the fact that
\[
E_2(0)I\left(x\sqrt{d}\right)=2E_2(0)x\sqrt{d},
\]
we have 
\[
\lim_{p\to\infty}\tilde D_L(1-kd,pe\sqrt{d})=2x. 
\]
Therefore, since $2x$ is a rational number, we can conclude that the Dedekind sums are dense in the rational numbers. This concludes the proof. 
\end{proof}

%\textbf{Additional Information}

Before concluding, it is of interest to mention that while our result extends Ito's in a certain sense, it remains open whether the graph $(a/c, \tilde D_L(a/c))$ is dense in $\mathbb C\times \mathbb R$ for general $L$. Indeed, in \cite{kohnen}, Kohnen also poses the question of whether his method can be used with classical Dedekind sums to prove the density statement for the graph. 

\begin{rem}
We are grateful to the anonymous referee for providing the following insightful comments.

\begin{enumerate}
\item
There is an alternative proof for Theorem 1 using properties of the group homomorphism $\Phi$ established in \cite{sczech}. It was shown in that paper that the set of values of $\Phi$ forms a subgroup of the additive group of complex numbers which has rank greater or equal to $h = h(\mathcal{O}_L)$ when viewed as a $\mathbb Z$-module. So when the normalized values are all real, then the values can not be all rational in the case $h > 1$. Some of these values must be irrational if $h > 1$. From this it follows easily that some of the normalized Dedekind sums must be irrational as well since the other term contributing to $\Phi(a,b,c,d)$ is rational (after normalization). From this it follows that the values of the normalized Dedekind sums are dense on the real line in the case of a lattice $L$ with a real $j$-invariant $j(L)$ and $h(\mathcal{O}_L) > 1$. A variation of this argument also shows the density of the normalized values of the Dedekind sums in the excluded case $h=1$. The key observation here is that the values $I((a+d)/c)/\sqrt{d}$ are dense in the field of rational numbers.\\

\item
The density result also has a geometric interpretation. If $\mathcal{O}_L$ is fixed, then one can chose a lattice $L$ in each similarity class of lattices with that fixed multiplier ring $O_L$. This gives a set of h homomorphisms $\Phi_L$. The set of these $h$ homomorphisms take values in $\mathbb{C}^h$. In fact, the set of values in $\mathbb{C}^h$ forms a discrete lattice. The situation is analogous to the case of the ring of algebraic integers in a real quadratic field. These integers are dense on the real line once a specific real embedding of the real quadratic field was chosen. However, if both embeddings are considered simultanously, the integers form a lattice in $\mathbb{R}^2$.
\end{enumerate}
\end{rem}

\subsubsection*{Acknowledgments} 
The authors are grateful to the anonymous referee for helpful feedback improving the paper. The authors thank Pierre Charollois and Lawrence Washington for comments on an earlier version of this article.
This research was completed at the REU Site: Mathematical Analysis and Applications at the University of Michigan-Dearborn, and was supported by the grants NSF DMS-1950102 and NSA H98230-21.
\nocite{*}
\bibliographystyle{plain}
\bibliography{sources}
\end{document}